\documentclass[]{article}

\usepackage{graphicx}
\usepackage{amsmath}
\usepackage{amssymb}
\usepackage{amsthm} 
\usepackage{bm}
\usepackage{pxfonts}
\usepackage{enumerate}
\usepackage{color}
\usepackage{mathdots}
\usepackage{sectsty}
\usepackage[hidelinks]{hyperref}
\usepackage{tikz}
\usepackage{caption}
\usepackage{adjustbox}
\usepackage{bbold}

\sectionfont{\scshape\centering\fontsize{11}{14}\selectfont}
\subsectionfont{\scshape\fontsize{11}{14}\selectfont}
\usepackage{fancyhdr}
\DeclareMathOperator{\Tr}{Tr}

\newcommand\shorttitle{Moments of Generalized Cauchy Random Matrices and continuous-Hahn Polynomials}
\newcommand\authors{T. Assiotis, B. Bedert, M. A. Gunes and A. Soor}

\newtheorem{thm}{Theorem}[section]
\newtheorem{cor}[thm]{Corollary}
\newtheorem{lem}[thm]{Lemma}

\newtheorem{rmk}[thm]{Remark}
\newtheorem{prop}[thm]{Proposition}

\title{\large \bf Moments of Generalized Cauchy Random Matrices and continuous-Hahn Polynomials}
\author{\small THEODOROS ASSIOTIS, BENJAMIN BEDERT, MUSTAFA ALPER GUNES AND ARUN SOOR}
\date{}

\begin{document}

\maketitle

\begin{abstract}
 In this paper we prove that, after an appropriate rescaling, the sum of moments $\mathbb{E}_{N}^{(s)} \left( \Tr \left( |\mathbf{H}|^{2k+2}+|\mathbf{H}|^{2k}\right) \right)$ of an $N\times N$ Hermitian matrix $\mathbf{H}$ sampled according to the generalized Cauchy (also known as Hua-Pickrell) ensemble with parameter $s>0$ is a continuous-Hahn polynomial in the variable $k$. This completes the picture of the investigation that began in \cite{CundenMezzadriOConnellSimm} where analogous results were obtained for the other three classical ensembles of random matrices, the Gaussian, the Laguerre and Jacobi. Our strategy of proof is somewhat different from the one in \cite{CundenMezzadriOConnellSimm} due to the fact that the generalized Cauchy is the only classical ensemble which has a finite number of integer moments. Our arguments also apply, with straightforward modifications, to the other three cases studied  in \cite{CundenMezzadriOConnellSimm} as well. We finally obtain a differential equation for the one-point density function of the eigenvalue distribution of this ensemble and establish the large $N$ asymptotics of the moments.
\end{abstract}

\tableofcontents

\section{Introduction}

The study of moments of traces of powers of random matrices has a long history, with many diverse connections including representation theory, enumerative combinatorics and $C^*$-algebras, see for example \cite{DiaconisRM}, \cite{Harer}, \cite{CundenDahlqvistOConnell},
\cite{HaagerupTj}. Recently a novel connection between moments of random matrices and hypergeometric orthogonal polynomials was discovered in \cite{CundenMezzadriOConnellSimm}. The results of \cite{CundenMezzadriOConnellSimm} informally read as follows: in the setting of three of the classical Hermitian unitarily invariant ensembles, the Gaussian (GUE), the Laguerre (LUE) and Jacobi (JUE), if one looks at moments of traces of powers of the matrices as functions of the exponents then these are essentially, up to some factor, hypergeometric orthogonal polynomials from the Askey scheme \cite{HypergeometricOrthogonalPolynomials}
.

The goal of this paper is to prove an analogous result for the fourth, and final\footnote{See for example \cite[Section 5.4.3]{Loggases} for more details on structural properties of the classical ensembles.}, classical random matrix ensemble, the generalized Cauchy (also known as Hua-Pickrell ensemble), thus completing the picture. The reader is referred to the following small sample \cite{Loggases}, \cite{Hua}, \cite{Pickrellmeasure}, \cite{Neretin}, \cite{BorodinOlshanskiErgodic}, \cite{Najdunel}, \cite{ForresterWitte}, \cite{ForresterWitte2}, \cite{Qiu} of the vast literature for more details on this ensemble. The generalized Cauchy ensemble is also of special interest because of its intimate connection, that we briefly recall below, see \cite{Neretin}, \cite{BorodinOlshanskiErgodic} for more details, to Haar distributed random unitary matrices under an application  of the Cayley transformation. In particular, the study of a somewhat different type of moments than the ones we consider here of generalized Cauchy random matrices has recently been instrumental in understanding the asymptotics of joint moments of characteristic polynomials of Haar distributed unitary matrices and their derivatives, see for example \cite{JointMoments}, \cite{JointMomentsPainleve}.

The generalized Cauchy ensemble is also rather exceptional in that it is the only one out of the four classical ensembles that has only a finite number of integer moments. This led us to take a somewhat different approach to prove our results compared to the one in \cite{CundenMezzadriOConnellSimm}. More precisely, two different proofs are given in \cite{CundenMezzadriOConnellSimm}. The first one involves analytic continuation from the expressions for integer moments using Carlson's Theorem. Clearly, this is not applicable to our case. The second proof in \cite{CundenMezzadriOConnellSimm} makes use of properties of the Mellin transform of the Wronskians of wavefunctions of the corresponding weight to obtain a three term recurrence for the moments. In our case the weight depends on the size N of the matrices involved and thus these wavefunctions come from different orthogonality weight families when we vary the matrix size N. This introduces certain complications in obtaining a closed recurrence, which despite significant efforts we were not able to resolve.

Thus, we were forced to find a somewhat different strategy of proof, which we will outline below after the statement of our main result Theorem \ref{MainTheorem}. We note that, each step in this slightly more involved proof can be adapted, with straightforward modifications, to the cases of the GUE, LUE and JUE. 

We now briefly discuss a couple of side results that we obtain in the course of our investigation along with possible applications and future directions. For the other three classical ensembles it is well-known, see \cite{CundenMezzadriOConnellSimm} and the references therein for more details, that the moments of traces of integer powers of the matrix are in fact polynomials or rational functions in the matrix size N. As a by-product of some of our intermediate results, that we need in order to prove our main theorem, we obtain that the (sum of consecutive even) moments of the generalized Cauchy ensemble that we study here are also polynomials in N. In a different direction, we obtain a differential equation for the one-point density of the generalized Cauchy ensemble. Analogous differential equations in the case of GUE and LUE were used in \cite{GotzeTikhomirov} to establish optimal rates of convergence towards the corresponding limiting densities. It would be of interest to investigate an analogous application in the Cauchy setting.

We note that the results of \cite{CundenMezzadriOConnellSimm} have recently been extended to multi-point correlators and connections to different types of Hurwitz numbers have been established, see \cite{DubrovinYang}, \cite{LaguerreCorrelators}, \cite{JacobiCorrelators}, \cite{CundenDahlqvistOConnell} for more details. Also, in a different direction, an extension of some of the results from \cite{CundenMezzadriOConnellSimm}, in the case of a high-dimensional analogue of the GUE related to the $d$-dimensional Fermi gas in a harmonic trap has been proven in \cite{ForresterMultidimensional}. It would be very interesting to investigate the analogous questions in the setting of the generalized Cauchy ensemble.

Finally, related questions about moments have been investigated for certain discrete ensembles in \cite{DiscreteEnsembles} where the situation is slightly more subtle compared to the continuous setting of random matrices, see \cite{DiscreteEnsembles} for more details. It is worth mentioning that there is a rather natural discrete analogue of the Cauchy ensemble, the so-called $zw$-measures which appear in the problem of harmonic analysis on the infinite-dimensional unitary group, see \cite{BorodinOlshanskiHarmonic}, \cite{BorodinOlshanskiMarkov}. It would be interesting to investigate properties such as the ones established in \cite{DiscreteEnsembles} for the other discrete ensembles for the $zw$-measures as well.

We now turn to the definitions of the generalized Cauchy ensemble and of the continuous-Hahn polynomials. These will allow us to state our main result precisely.

\paragraph{Generalized Cauchy Ensemble.} Let $\mathbb{H}(N)$ denote the linear space of $N \times N$ complex Hermitian matrices. We define, for a parameter $s\in \mathbb{R}$, $ s>-\frac{1}{2}$, the generalized Cauchy probability measure on $\mathbb{H}(N)$ as follows:
\begin{equation}
   \frac{1}{F(s,N)}\cdot \det\left(\left(1+\mathbf{X}^2\right)^{-s-N}\right) \times d\mathbf{X}, \label{eq:1}
\end{equation}
where $d\mathbf{X}$ is the Lebesgue measure on $\mathbb{H}(N)$ and $F(s,N)$ is given explicitly by:
\begin{equation}
    F(s,N)=\prod_{j=1}^N \frac{\pi^j\Gamma(2s+j)}{2^{2s+2j-2}{\Gamma(s+j)}^2}.
\end{equation}
Using the Cayley transform 
\begin{equation}
  \mathbf{X} \mapsto \mathbf{U}=\frac{i-\mathbf{X}}{i+\mathbf{X}} \in \mathbb{U}(N),
  \label{cayley}
\end{equation}
we can transform this measure to the following probability measure on the group of $N\times N$ unitary matrices $\mathbb{U}(N)$:
\begin{equation}
    \frac{1}{G(s,N)}\cdot \det\left(\left(1+\mathbf{U}\right)^{2s}\right)\times d\mathbf{U}, \label{eq:haarmeasure}
\end{equation}
where $d\mathbf{U}$ denotes the Haar probability measure on $\mathbb{U}(N)$ and:
\begin{equation}
    G(s,N)=\frac{1}{N!}\prod_{j=1}^N \frac{\Gamma(2s+j)\Gamma(j+1)}{\Gamma(s+j)^2}.
\end{equation}

\paragraph{Continuous-Hahn Polynomials.} We use the standard definition of hypergeometric functions:
\begin{equation}
 _pF_q \left[\begin{matrix}
   a_1,\ldots, a_p \\
    b_1,\ldots, b_q \end{matrix} \ ; \ z\right]=\sum_{k=0}^{\infty} \frac{(a_1)_k\ldots(a_p)_k}{(b_1)_k\ldots(b_q)_k}  \frac{z^k}{k!},
    \label{eq:hypergeom}
\end{equation}
where $(x)_k=\frac{\Gamma(x+k)}{\Gamma(x)}$. We then define the continuous-Hahn polynomials as:
\begin{equation}
S_n(x;a,b,c,d)=
i^n\frac{{(a+c)_n}{(a+d)_n}}{n!} {_3F_2 \left[\begin{matrix}
   -n,n+a+b+c+d-1,a+ix \\
    a+c,a+d \end{matrix} \ ; \ 1\right]}.
\label{eq:2}
\end{equation}
When $\Re(a,b,c,d)>0$, $\bar a=c$ and $\bar b=d$ these polynomials are orthogonal on $\mathbb{R}$ with respect to the weight $\Gamma(a+ix)\Gamma(b+ix)\Gamma(c-ix)\Gamma(d-ix)$ (see \cite[§ 9.4]{HypergeometricOrthogonalPolynomials}).\\

We can now state the main result of this paper.

\begin{thm}
Let $s \in \mathbb{R}$ and $s>0$. Let $\mathbf{H} \in \mathbb{H}(N)$ and define the sum of moments
\begin{equation}
\mathsf{Q}(k;s,N)=\mathbb{E}_{N}^{(s)} \left( \Tr \left( |\mathbf{H}|^{2k+2}+|\mathbf{H}|^{2k}\right) \right)
\label{Qdefntn}
\end{equation}
where $\mathbb{E}_{N}^{(s)}$ denotes expectation with respect to the generalized Cauchy measure \eqref{eq:1} so that $\mathsf{Q}(k;s,N)$ exists when $-\frac{1}{2}<\Re\left( k\right)<s-\frac{1}{2}$. Then if we let $k=ix-1$, we have that:
\begin{align}
 \mathsf{Q}(k;s,N)=\frac{\Gamma\left(k+\frac{1}{2}\right) \Gamma\left(s-k-\frac{1}{2}\right)}{\Gamma\left(s+\frac{3}{2}\right)\Gamma\left(-s-\frac{1}{2}\right) \sqrt{\pi}} \frac{i^{1-N}}{2} \Gamma\left(\frac{1}{2}-s-N\right)s(2s+N) \times S_{N-1}\left(x;1,\frac{1}{2}+s,1,\frac{1}{2}+s\right),
\end{align}
where $S_n(x;a,b,c,d)$ denotes the continuous-Hahn polynomial with parameters $a,b,c,d$ and $n$. In particular, $\frac{\mathsf{Q}(k;s,N)}{\Gamma\left(k+\frac{1}{2}\right) \Gamma\left(s-k-\frac{1}{2}\right)} $ can be extended to an analytic function in $\mathbb{C}$ that is invariant up to a change of sign under reflection $k\rightarrow -k-2$, with all complex zeros lying on the vertical line $\Re(k)=-1$.
\label{MainTheorem}
\end{thm}

The strategy of proof of Theorem \ref{MainTheorem} goes as follows. It is relatively straightforward to show first that after suitable rescaling $\mathsf{Q}(k;s,N)$ is indeed a polynomial in $k$ of degree $N-1$. Then, as a first step towards identifying this polynomial we obtain a difference equation satisfied by the rescaled version of $\mathsf{Q}(k;s,N)$. This is done by using a method due to Ledoux from \cite{Ledoux}. Here, we initially need to restrict to $s>3$ because we require a certain number of moments in order to apply Ledoux's method (this restriction is removed at the very end of the proof by an analytic continuation argument). We then show that this particular difference equation identifies polynomial solutions up to a multiplicative constant (in fact, analogous statements can be proven for the Gaussian, Laguerre and Jacobi ensembles studied in \cite{CundenMezzadriOConnellSimm}). Finally, by putting everything together we obtain Theorem \ref{MainTheorem}.

\paragraph{Organisation of the paper.} The rest of this paper is organised as follows: In Section 2, we introduce the pseudo-Jacobi polynomials and related differential identities. Then, in Section 3, we establish the difference equation satisfied by $\mathsf{Q}(k;s,N)$ after rescaling. In Section 4, we prove that the difference equation obtained in Section 3 determines a polynomial uniquely up to multiplication by a constant and complete the proof of our main result. In Section 5, we obtain a differential equation for the one-point density function of the generalized Cauchy ensemble. Finally, in Section 6 we discuss the large-$N$ asymptotics of the moments.
\paragraph{Acknowledgements}
A. Soor and B. Bedert are grateful for the financial support from Mathematical Institute, University of Oxford. B. Bedert is also grateful for the financial support from the EPSRC. M. A. Gunes is grateful for financial support from Prof. Jon Keating's start-up grant. T. Assiotis is grateful for financial support at the early stages of this work from ERC Advanced Grant  740900 (LogCorRM). We are very grateful to two anonymous referees and an associate editor for a number of comments and suggestions which have improved the presentation.

\section{The pseudo-Jacobi Ensemble}
The measure in \eqref{eq:1} can be defined more generally for $s\in\mathbb{C}$ with $\Re{s}>-\frac{1}{2}$ as follows (this generalization is due to Neretin in \cite{Neretin}):
\begin{equation}
    \frac{1}{\widetilde{F}(s,N)}\cdot \det\left(\left(1+i\mathbf{X}\right)^{-s-N}\right)\det\left(\left(1-i\mathbf{X}\right)^{-\overline{s}-N}\right) \times d\mathbf{X}
\end{equation}
where $d\mathbf{X}$ is the Lebesgue measure on $\mathbb{H}(N)$ and
\begin{equation}
    \widetilde{F}(s,N)=\prod_{j=1}^N \frac{\pi^j\Gamma(2\Re(s)+j)}{2^{2\Re(s)+2j-2}{\Gamma(s+j)}{\Gamma\left(\overline{s}+j\right)}}.
\end{equation}
We continue to denote the expectations with respect to this measure by $\mathbb{E}_{N}^{(s)}$. The eigenvalue density for this ensemble is  given by the following probability measure on $\mathbb{R}^N/\mathfrak{S}(N)$ where $\mathfrak{S}(N)$ is the $N$-th symmetric group; which gives rise to a determinantal point process on $\mathbb{R}$ with $N$ points (see \cite{BorodinOlshanskiErgodic}):
\begin{align}
     \frac{1}{T(s,N)}\cdot \prod _{1\leq l<k\leq N}(x_k-x_l)^2 \prod _{j=1}^{N}\left( 1+x_{j}^{2}\right) ^{-\Re\left( s\right) -N }e^{2\Im(s)\tan ^{-1}(x_{j})}dx_j
     \label{eq:huapickrellensemble}
\end{align}
where the normalisation constant $T(s,N)$ can be computed explicitly as:
\begin{equation}
   T(s,N)=\pi^N2^{-N(N+2\Re s -1)} \cdot \prod_{j=0}^{N-1} \frac{j!\Gamma(2\Re s + N-j)}{\Gamma(s+N-j)\Gamma(\overline{s}+N-j)} 
\end{equation}
(see \cite{Selberg}).
We note that this is a finite orthogonal polynomial ensemble called the "pseudo-Jacobi Ensemble" (see \cite{Konig}) corresponding to the weight
\begin{equation*}
\phi_N^{(s)} \left( x\right) =\left( 1+x^{2}\right) ^{-N-\Re(s)} \cdot e^{2\Im(s)\tan ^{-1}(x)}.
\end{equation*}
We let ${p_0^{(s,N)},p_1^{(s,N)},\ldots}$ denote the monic orthogonal polynomials associated to the weight function $\phi_N^{(s)} \left( x\right)$ on $\mathbb{R}$ where the $m$-th pseudo-Jacobi polynomial $p_m^{(s,N)}$ exists if $m<\Re(s)+N-\frac{1}{2}$ and is given explicitly by (see \cite[§1]{BorodinOlshanskiErgodic}):
\begin{equation}
    p_m^{(s,N)}(x) := (x-i)^m\  _2F_1\left[\begin{matrix}
    -m, s+N-m \\
    2\Re(s) + 2N - 2m   
    \end{matrix} ; \frac{2}{1+ix}\right],
\label{eq:pseudojacobipolynomials}
\end{equation}
from which the following parity property can be deduced:
\begin{equation}
    p_m^{(s,N)}(-x)=(-1)^m p_m^{(\bar s,N)}.
\end{equation}
We also note that the eigenvalue density corresponding to the measure in \eqref{eq:haarmeasure} is given by the following probability measure on $\mathbb{T}^N/\mathfrak{S}(N)$ where $\mathbb{T}$ denotes the unit circle, which gives rise to a  determinantal point process on $\mathbb{T}$ with $N$ points :
\begin{align}
     \frac{1}{Y(s,N)}\cdot \prod _{1\leq l<k\leq N}{\vert e^{i\theta_l}-e^{i\theta_k}\vert}^2  \prod _{j=1}^{N}\left( 1+e^{i\theta_j}\right)^{\bar s} \left( 1+e^{-i\theta_j}\right)^{s}d{\theta}_j,  &&  \theta_j \in [0,2\pi)
\label{eq:unitcircledensity}
\end{align}
where: 
\begin{equation}
    Y(s,N)=N! (2\pi)^N \prod_{j=1}^N \frac{\Gamma(2\Re(s)+j)\Gamma(j+1)}{\Gamma(s+j)\Gamma(\overline{s}+j)}.
\end{equation}

We now define the \textit{one-point density function} $ \bm{\rho}_N^{(s)} \left( x\right)$ corresponding to \eqref{eq:huapickrellensemble} as:
\begin{equation}
    \bm{\rho}_N^{(s)} \left( x\right)=\mathbb{E}_{N}^{(s)}\left(\sum_{j=1}^N\delta(x-x_j)\right)
    \label{eq:densitydef}
\end{equation}
where $\delta$ is the Dirac $\delta$ function. By a standard argument from Random Matrix Theory, (see \cite[§1]{CundenMezzadriOConnellSimm} for example)  $\bm{\rho}_N^{(s)} \left( x\right)$ is given explicitly by:
\begin{equation*}
    \bm{\rho}_N^{(s)} \left( x\right)=\phi_N^{(s)}(x)\sum_{j=0}^{N-1} \frac{\left(p_j^{(s,N)}\right)^{2}(x)}{\|{p^{(s,N)}_{j}}\|^2}.
\end{equation*}
Hence, by the Christoffel-Darboux formula, the one-point density function is given as follows: 
 \begin{equation}
\bm{\rho}_N^{(s)} \left( x\right) =\left[ p_{N-1}^{(s,N)}\left(p_{N}^{(s,N)}\right)'-p_{N}^{(s,N)}\left(p_{N-1}^{(s,N)}\right)'\right](x) \phi_N^{(s)} \left( x\right)  \cdot \gamma _{N-1,s}^{2},
\label{eq:onepointdensity}
\end{equation}
where (see \cite[§1]{BorodinOlshanskiErgodic}):
\begin{equation*}
\gamma_{N-1,s}^{2} =\frac{1}{{\left\|{p_{N-1}^{(s,N)}}\right\|}^2}= \frac {2^{2\Re\left( s\right) }} {\pi }\Gamma \left[ \begin{matrix} 2\Re\left( s\right) +N+1,s+1,\overline {s}+1\\ N,2\Re\left( s\right) +1,2\Re\left( s\right) +2\end{matrix} \right];
\end{equation*}
here, and in the remainder of the paper, we are using $\Gamma \left[ \begin{matrix} a_1, & a_2, & \dots \\ b_1, & b_2, & \dots \end{matrix} \right]$ to denote $\frac{\Gamma(a_1)\Gamma(a_2) \cdots }{\Gamma(b_1)\Gamma(b_2)\cdots}$. From \eqref{eq:onepointdensity} we get that: 
\begin{equation}
\frac{d}{dx} \left( \frac {\bm{\rho}_N^{(s)}(x) } {\gamma _{N-1,s}^{2} \phi_N^{(s)}(x)}\right)=\left[\frac{d^2p_{N}^{(s,N)}}{dx^2}p_{N-1}^{(s,N)}\right](x)-\left[p_{N}^{(s,N)}\frac{d^2p_{N-1}^{(s,N)}}{dx^2}\right](x). \label{eq:7}
\end{equation}
We also know from \cite{BorodinOlshanskiErgodic} that the polynomial $p^{(s,N)}_m$ satifies the differential equation:
\begin{align}
 -\left( 1+x^{2}\right) \left(p_{m}^{(s,N)}\right)''
+2(-\Im\left( s\right) +(\Re\left(s)+N-1\right) x) \left(p_{m}^{(s,N)}\right)'\nonumber\\
+m\left( m+1-2\Re(s\right) -2N)p_{m}^{(s,N)}=0. \label{EigenEq}
\end{align}

 Substituting this differential equation into \eqref{eq:7}, we get:
\begin{equation*}
\begin{split}
\frac{d}{dx} \left( \frac {\bm{\rho}_N^{(s)}(x) } {\gamma _{N-1,s}^{2} \phi_N^{(s)}(x)}\right) =\frac {-2\left( \Im\left( s\right) +\left( 1-N-\Re\left( s\right) \right) x\right) } {1+x^{2}} \left( \frac {\bm{\rho}_N^{(s)}(x) } {\gamma _{N-1,s}^{2} \phi_N^{(s)}(x)}\right)-\frac {2\Re\left( s\right) \left[p_{N}^{(s,N)}p_{N-1}^{(s,N)}\right](x)} {\left( 1+x^{2}\right) }.
\end{split}
\end{equation*}
Hence, we get the equation:
\begin{equation}
\frac {d} {dx}\left( \left(1+x^{2}\right) \bm{\rho}_N^{(s)} \left( x\right) \right)=-2\Re(s) \gamma_{N-1,s}^2 \phi_N^{(s)} (x) p_{N}^{(s,N)}(x)p_{N-1}^{(s,N)}(x).
\label{diffrho}
\end{equation}

	\section{The difference equation for the moments}
In this section we obtain a three-term recurrence for $\mathsf{Q}(k;s,N)$ as defined in \eqref{Qdefntn}. This is one of the main steps towards proving the representation in terms of continuous-Hahn polynomials.

Let $\mathbf{H}$ be an $N \times N$ Hermitian matrix and consider the following moments with respect to the generalized Cauchy ensemble as given in \eqref{eq:1} with parameters $s,N$:
\begin{equation}
\mathsf{Q}(k;s,N)\coloneqq\mathbb{E}_{N}^{(s)} \left( \Tr \left( |\mathbf{H}|^{2k+2}+|\mathbf{H}|^{2k}\right) \right) =\int _{\mathbb{R} }|x|^{2k}\left( 1+x^{2}\right) \bm{\rho}_N^{(s)} \left( x\right) dx,
\label{eq:defnQ}
\end{equation}
where we obtain the second equality by definition of the one-point density function in \eqref{eq:densitydef}.

Then, when $s\in\mathbb{R}$ (in which case the density is symmetric around $0$) with $s>0$, on the strip $-\frac{1}{2}<\Re(k)<s-\frac{1}{2}$, using integration by parts and equation \eqref{diffrho}, we get that:
\begin{equation}
\mathsf{Q}(k;s,N) = 4s \gamma^2_{N-1,s} \int _{\mathbb{R^{+}} }\frac {x^{2k+1}} {2k+1}p_{N}^{(s,N)}(x)p_{N-1}^{(s,N)}(x)\phi_N^{(s)} \left( x\right) dx,
\label{Qtoa(k)}
\end{equation}
where $p_N^{(s,N)}$ and $p_{N-1}^{(s,N)}$ are the polynomials that come from the pseudo-Jacobi ensemble with parameters $s,N$.

Our strategy is to use a method due to Ledoux from \cite{Ledoux}, and obtain a recurrence relation for the following function of $t\in\mathbb{R}$:
\begin{align}
a(t)= a(t;s,N) =\int _{\mathbb{R^{+}}} x^{t}p_{N}^{(s,N)}(x)p_{N-1}^{(s,N)}(x)\phi_N^{(s)} \left( x\right) dx.
\label{defa(k)}
\end{align}
We also observe that from (\ref{Qtoa(k)}) we have:
\begin{align}
 a(2k+1)=\frac{1}{4s\gamma_{N-1,s}^2} \cdot (2k+1)\mathsf{Q}(k;s,N),
 \label{Qarelation}
 \end{align} 
  which will allow us to write this recurrence in terms of $\mathsf{Q}$.
  
For the reader's convenience, we state explicitly below a proposition that we extract from the results of \cite{Ledoux} that will give us the desired recurrence for $a(t)$. For the statement of this proposition we need to introduce a number of quantities that will be computed explicitly in the sequel. The key ingredient is the following differential operator, acting on smooth functions $f$:
\begin{equation*}
\mathsf{L}\left( f\right) =\left( 1+x^{2}\right) f''\\+\left( \alpha +2\beta x\right) f',
\end{equation*}
where in order to ease notation we write, for a fixed $s\in \mathbb{C}$:
\begin{equation*}
    \alpha=\alpha(s)=2\Im(s), \: \beta=\beta(s)=1-\Re(s)-N.
\end{equation*}
We note that this operator can also be thought of as the generator of a one-dimensional diffusion process which plays a key role in \cite{HuaPickrellDiffusions}, \cite{MatrixBougerol}. We also need to consider the corresponding carré-du-champ (as known in the Markov process literature, see Section 2 in \cite{Ledoux}) operator $\Gamma$:
\begin{align*}
\Gamma \left( f,g\right) =\frac {1} {2}( \mathsf{L}\left( fg\right) -g\mathsf{L}\left( f\right) -f\mathsf{L}\left( g\right))=\left( 1+x^{2}\right) f'g'.
\end{align*}
Finally, we note that as we recalled in (\ref{EigenEq}) in the previous section, see \cite{BorodinOlshanskiErgodic}, the polynomials $p_m^{(s,N)}$ are actually eigenfunctions of $\mathsf{L}$.
\begin{prop}
Let the operators $\mathsf{L}$ and $\Gamma$ be as above. Consider the function $A(x)$ given by:
\begin{equation*}
  A\left( x\right)=\left( -\mathsf{L}\right) \left( x\right) 
\end{equation*}
and let $B(x)$ be the function such that for any smooth functions $f,g$:
\begin{equation*}
    \Gamma\left(x,f\right)\Gamma\left(x,g\right)=B(x)\Gamma\left(f,g\right).
\end{equation*}
Let $\bm{\tau}_{p_N^{(s,N)}}$ and $\bm{\tau}_{p_{N-1}^{(s,N)}}$ denote the eigenvalues of  $-\mathsf{L}$ corresponding to $p_N^{(s,N)}$ and $p_{N-1}^{(s,N)}$. Let $d_{p_N^{(s,N)}}, d_{p_{N-1}^{(s,N)}}$ be constants and $D_{p_N^{(s,N)}}(x), D_{p_{N-1}^{(s,N)}}(x)$ be functions such that:
\begin{align*}
    \left( -\mathsf{L}\right) \left( \Gamma \left( x,p_{N}^{(s,N)}\right)\right)(x) &= d_{p_N^{(s,N)}} \Gamma \left( x,p_{N}^{(s,N)}\right) + D_{p_N^{(s,N)}}(x)p_N^{(s,N)}(x),\\ \left( -\mathsf{L}\right) \left( \Gamma \left( x,p_{N-1}^{(s,N)}\right)\right)(x) &= d_{p_{N-1}^{(s,N)}} \Gamma \left( x,p_{N-1}^{(s,N)}\right) + D_{p_{N-1}^{(s,N)}}(x)p_{N-1}^{(s,N)}(x).
\end{align*}
Finally, we let: 
\begin{align*}
    u&=\frac{1}{2}\left(\bm{\tau}_{p_N^{(s,N)}}-\bm{\tau}_{p_{N-1}^{(s,N)}}\right)\left(d_{p_N^{(s,N)}}- d_{p_{N-1}^{(s,N)}}+\bm{\tau}_{p_N^{(s,N)}}-\bm{\tau}_{p_{N-1}^{(s,N)}}\right),\\
    v&=\frac{1}{2}\left(d_{p_N^{(s,N)}}+ d_{p_{N-1}^{(s,N)}}-\bm{\tau}_{p_N^{(s,N)}}-\bm{\tau}_{p_{N-1}^{(s,N)}}\right).
\end{align*}
For a function $f$ we write (assuming the integral exists):
\begin{equation*}
    G(f)=G^{(s,N)}(f)=\int_{\mathbb{R}^+} f(x) p_N^{(s,N)} p_{N-1}^{(s,N)} \phi_N^{(s)}(x) dx.
\end{equation*} 
Assume that $\Re(s)>\frac{5}{2}$ and $t\in \left(3,2\Re(s)-2\right)$ and consider the function $\theta:\mathbb{R}^+\to \mathbb{R}^+$ given by $\theta(x)=\theta_t(x)=x^t$. Then, we have the following equality:
\begin{equation}\label{LedouxEquality}
    G\left(M_4 \theta^{(4)}\right)+G\left(M_3 \theta^{(3)}\right)+G\left(M_2 \theta^{(2)}\right)+G\left(M_1 \theta^{(1)}\right)+G\left(M_0 \theta\right)=0,
\end{equation}
where $f^{(n)}$ denotes the $n$-th derivative of a function $f$ and:
\begin{align*}
  & M_4(x)=-B^2(x), \\
  & M_3(x)=-3B(x)B'(x),\\
  & M_2(x)=A^2(x)-A'(x)B(x)+2A(x)B'(x)+vB(x)-2\left(\bm{\tau}_{p_N^{(s,N)}}+\bm{\tau}_{p_{N-1}^{(s,N)}}\right)B(x) - 2B(x)B''(x),\\
  & M_1(x)=-vA(x)-D_{p_N^{(s,N)}}(x)-D_{p_{N-1}^{(s,N)}}(x),\\
  & M_0(x)=-u.
\end{align*}
\begin{proof}
This is essentially an immediate consequence of the computations done for Proposition 3.1 and Corollary 3.2 from \cite{Ledoux} applied to our operator $\mathsf{L}$ and function $\theta(x)=\theta_t(x)$. These results are proven by a direct calculation using integration by parts and the restriction on the parameter $t$ is simply so that all integrals involved in the computation are finite.
\end{proof}
\label{Ledouxprop}
\end{prop}
Our goal is to apply the proposition above to obtain a difference equation\footnote{As we see below the functions $M_i(x)$ are actually polynomials and thus equality (\ref{LedouxEquality}) reduces to a difference equation for $a(t)$.} for $a(t)$ of the form (\ref{generalrec}) below. Making use of (\ref{Qarelation}) we will then obtain a difference equation for $\mathsf{Q}$ as in (\ref{eq:13}) below when $s$ is real (a simplification occurs in this case). We now compute all the relevant quantities from Proposition \ref{Ledouxprop}.

As recalled in (\ref{EigenEq}) in the previous section, we know that:
\begin{align*}
\bm{\tau}_{p_N^{(s,N)}} =-N \left( N+2\beta -1\right) , \: \bm{\tau}_{p_{N-1}^{(s,N)}} =-(N-1) \left( N+2\beta -2\right).
\end{align*}
By a straightforward calculation, we also obtain:
\begin{equation*}
\begin{aligned}
 &A\left( x\right) =-\left( \alpha +2\beta x\right), \\
&B\left( x\right) =\left( \begin{matrix} 1 +x^{2}\end{matrix} \right).
\end{aligned}
\end{equation*}
Now, to find the remaining functions and constants, we use the ODE for the pseudo-Jacobi polynomials (and the 3rd order ODE obtained by differentiating the original ODE) to obtain:
\begin{align*}
\left( -\mathsf{L}\right) \left( \Gamma \left( x,p_N^{(s,N)}\right) \right)(x) ={}& (1+x^2){p_N^{(s,N)}}'(x) (2\beta -2 -N(N+2\beta -1))-2x(N(N+2\beta-1))p_N^{(s,N)}(x),\\
\left( -\mathsf{L}\right) \left( \Gamma \left( x,p_{N-1}^{(s,N)}\right) \right)(x) ={}&  (1+x^2){p_{N-1}^{(s,N)}}'(x) (2\beta -2 -(N-1)(N+2\beta -2))\\ &-2x((N-1)(N+2\beta-2))p_{N-1}^{(s,N)}(x).
\end{align*}
Hence, we arrive at the following expressions for the remaining constants and functions:
\begin{align*}
d_{p_N^{(s,N)}} &=2\beta-2-N(N+2\beta-1)=2\beta-2+\bm{\tau}_{{p_N}^{(s,N)}}, \\
D_{p_N^{(s,N)}}(x) &=-2x(N(N+2\beta-1))=2x\bm{\tau}_{p_N^{(s,N)}}, \\
 d_{p_{N-1}^{(s,N)}} &=2\beta-2-(N-1)(N+2\beta-2)=2\beta-2+\bm{\tau}_{p_{N-1}^{(s,N)}}, \\  
D_{p_{N-1}^{(s,N)}}(x) &=-2x(N-1)(N+2\beta-2)=2x\bm{\tau}_{p_{N-1}^{(s,N)}}.
\end{align*}

\begin{flushleft}
Now, we are ready to prove the following key proposition.
\end{flushleft}

\begin{prop}\label{DiffEqProp} Let $s\in\mathbb{R}$ and $s>3$, then, when $k \in \left(\frac{1}{2},s-\frac{5}{2}\right)$, the sum of moments of generalized Cauchy matrices $\mathsf{Q}(k;s,N)$ as defined in \eqref{eq:defnQ} satisfies the following difference equation
\begin{equation}
R(k)\mathsf{Q}(k+1;s,N)+T(k)\mathsf{Q}(k;s,N)+S(k)\mathsf{Q}(k-1;s,N)=0
\label{eq:13}
\end{equation}
with coefficients
\begin{align*}
R(k)&=(2k+4)(4s^{2}-(2k+3)^2),\\
T(k)&=-2(2k+1)\left(2N(N+2s)+(2k+2)^2\right),\\
S(k)&=-(2k+1)2k(2k-1),
\end{align*}
and initial conditions given by
\begin{align*}
\mathsf{Q}(0;s,N)&=\frac{2sN(2s+N)}{(2s-1)(2s+1)},\\
\mathsf{Q}(1;s,N)&=\frac{2sN(2s+N)}{(2s-1)(2s+1)} \frac{2Ns+N^2+2}{(2s+3)(2s-3)}=\mathsf{Q}(0;s,N)\cdot\frac{2Ns+N^2+2}{(2s+3)(2s-3)}.
\end{align*}
\end{prop}
\begin{proof}
By making use of Proposition \ref{Ledouxprop} and after some elementary computations we obtain the following recursion for $a(t)$, where we consider more generally $s\in \mathbb{C}$:
\begin{equation}
    C_1(t)a(t)+C_2(t)a(t-1)+C_3(t)a(t-2)+C_4(t)a(t-3)+C_5(t)a(t-4)=0,
    \label{generalrec}
\end{equation}
where the coefficients are given by:
\begin{align*}
    C_1(t) &= 4\Re\left( s\right)^2(t+1)(t-1) - 6t(t-1)^2 - t(t-1)(t-2)(t-3), \\
    C_2(t) &= 4\Im\left( s\right)(-\Re\left( s\right)-N)t(2t-1), \\
    C_3(t) &= t(t-1)\left[-2(t-1)^2+4\Im\left( s\right)^2+4N(-N-2\Re\left( s\right))\right], \\
    C_4(t)&=0,\\
    C_5(t) &= -t(t-1)(t-2)(t-3).
\end{align*}
We note that when $s\in \mathbb{R}$, $C_2(t)$ vanishes so that, similarly to the GUE and LUE cases considered in \cite{Ledoux}, this reduces to a three-term recurrence. Using \eqref{Qarelation}, with the correspondence $t=2k+3$, gives the desired recurrence for $\mathsf{Q}(k;s,N)$ in the given range\footnote{For our purposes, we need the difference equation in \eqref{eq:13} to hold on an interval of real numbers, which is possible when a recurrence of the form \eqref{generalrec} can be obtained for $t>4$, which implies that $s>3$ is needed.} of values for $k$. The initial conditions $\mathsf{Q}(0;s,N)$ and $\mathsf{Q}(1;s,N)$ are obtained by simply computing $a(1)$ and $a(3)$ by using the three-term recurrence formula for the monic pseudo-Jacobi polynomials and orthogonality.

\end{proof}

\begin{rmk}
For fixed $k\in \mathbb{N}$, under the restrictions on $s,k$ from Proposition \ref{DiffEqProp}, it is a direct consequence of the recurrence (\ref{eq:13}) by induction that the function $N \mapsto   \mathsf{Q}(k;s,N)$
is actually a polynomial with coefficients which are rational functions in $s$. Then, an analytic continuation argument analogous to the one presented in the proof of Theorem \ref{MainTheorem} in Section 4 extends this statement to the optimal range $s>\frac{1}{2}$ and $k\in \big\{0,1,2, \dots, \big\lfloor s-\frac{1}{2} \big\rfloor \big\}$.
\end{rmk}

\begin{rmk}
We note that, the difference equation (\ref{generalrec}) allows us to deduce a recurrence for the integer moments of the generalized Cauchy matrices for complex values of s as well. Namely, define for an integer $m$ with $0\leq m <2s-1$ the following sum of moments \footnote{When $s$ is complex the weight function is no longer symmetric around $x=0$ so that to apply Ledoux's method that we presented here we need to consider powers of the matrices themselves without the absolute values.}
\begin{equation}
\widetilde{\mathsf{Q}}(m;s,N)=\mathbb{E}_{N}^{(s)} \left( \Tr \left( \mathbf{H}_{N}^{m+2}+\mathbf{H}_{N}^{m}\right) \right).
\label{eq:defnQtilde}
\end{equation}
Then integrating by parts in a similar way as equation \eqref{Qtoa(k)} gives
\begin{equation}
    \widetilde{\mathsf{Q}}(m;s,N) = 2\Re(s)\gamma_{N-1,s}^2\cdot\frac{1}{m+1}\left[a(m+1;s,N)+(-1)^ma(m+1;\overline{s},N)\right]
\end{equation}
where $\overline{s}$ denotes the complex conjugate of s. Hence, noting that the coefficients $C_1$, $C_3
$ and $C_4$ are invariant under $s\leftrightarrow \overline{s}$ whereas $C_2$ changes sign, we immediately deduce the following recurrence for $\widetilde{\mathsf{Q}}(m;s,N)$
\begin{equation}
    D_1(m)\widetilde{\mathsf{Q}}(m;s,N)+D_2(m)\widetilde{\mathsf{Q}}(m-1;s,N)+D_3(m)\widetilde{\mathsf{Q}}(m-2;s,N) + D_4(m)\widetilde{\mathsf{Q}}(m-4;s,N) = 0
\label{tildediff}
\end{equation}
where
\begin{align*}
    D_1(m) &= (m+2)(4\Re\left( s\right)^2 - (m+1)^2), \\
    D_2(m) &= 4\Im\left( s\right)(-\Re\left( s\right)-N)(2m+1), \\
    D_3(m) &= (m-1)\left[-2m^2+4\Im\left( s\right)^2+4N(-N-2\Re\left( s\right))\right], \\
    D_4(m) &= -(m-1)(m-2)(m-3).
\end{align*}
Finally, one may note that equation \eqref{tildediff} reduces to \eqref{eq:13} by taking $s$ real and substituting $m = 2k+2$.
\label{rmk:3.2}
\end{rmk}

\begin{cor}Let $s\in\mathbb{R}$ and $s>0$,
and for $k$ in the strip $-\frac{1}{2}<\Re(k)<s-\frac{1}{2}$ define
\begin{align*}
\mathsf{J}(k;s,N)=\frac{\mathsf{Q}(k;s,N)}{\Gamma\left(k+\frac{1}{2}\right) \Gamma\left(s-k-\frac{1}{2}\right)}. 
\end{align*}
Then when $s>3$, $\mathsf{J}(k;s,N)$ satisfies the following difference equation for $k \in \left(\frac{1}{2},s-\frac{5}{2}\right)$:
\begin{multline}
(2k+4)(2s+2k+3)\mathsf{J}(k+1;s,N)+2k(2k+1-2s)\mathsf{J}(k-1;s,N) \\ = 2\left(2N(N+2s)+(2k+2)^2\right)\mathsf{J}(k;s,N). \label{eqJ}
\end{multline}
\end{cor}
\begin{proof}
Immediate using equation \eqref{eq:13}.
\end{proof}

\section{Proof of the main result}
We begin with the following lemma.
\begin{lem}Let $s\in\mathbb{R}$ and $s>3$, then for $k \in \left(-\frac{1}{2},s-\frac{1}{2}\right)$,\ $\mathsf{J}(k;s,N)$ is a polynomial in $k$ of degree $N-1$.
\label{poly1lemma}
\end{lem}
\begin{proof}
We note that when the parameter $s\in \mathbb{R}$ we have ${p_N^{(s,N)}}(-x){p_{N-1}^{(s,N)}}(-x)=-{p_N^{(s,N)}}(x){p_{N-1}^{(s,N)}}(x)$ so that ${p_N^{(s,N)}}{p_{N-1}^{(s,N)}}$ is a polynomial with terms of odd degree only. We can then compute for an odd integer $l$
\begin{align}
\int _{\mathbb{R^{+}} }\frac {x^{2k+1+l}} {(1+x^2)^{s+N}}dx= \frac{\Gamma\left(k+1+\frac{l}{2}\right) \Gamma\left(s+N-k-1-\frac{l}{2}\right)}{2\Gamma\left(s+N\right)}
\end{align}
as in equation (18) on page 953 in \cite{Polyanin}.
Plugging this integral back into the expression for $\mathsf{J}(k;s,N)$, noting the equality in \eqref{Qtoa(k)}, we see that $\mathsf{J}(k;s,N)$ is a linear combination of polynomials of degree $N-1$ in the variable $k$ and hence a polynomial of degree $N-1$ itself.
\end{proof}

Next we will see that polynomial solutions to \eqref{eqJ} are in fact uniquely determined up to multiplication by a constant, which allows us to deduce Theorem \ref{MainTheorem} by simply matching up two such solutions using a constant\footnote{In fact, it is possible to circumvent the use of Lemma \ref{poly2lemma} by fixing $N$, taking $s$ large enough with respect to $N$ and arguing differently. Then, the $N$-dependent restriction on $s$ can be removed by analytic continuation at the end. Nevertheless, we believe that both the statement, in particular because it applies uniformly in $N$, and proof of Lemma \ref{poly2lemma} are of independent interest and so we present it here.}. Note that the following result is not unique to the generalized Cauchy ensemble. As already mentioned in the introduction, one can use analogous arguments to show that the difference equations satisfied by appropriately rescaled moments of the Gaussian and Laguerre ensembles, and an appropriately rescaled difference of moments for the Jacobi ensemble (these are given explicitly in \cite{CundenMezzadriOConnellSimm}), also admit unique polynomial solutions up to multiplication by a constant. 

\begin{lem} Let $s>0$. Let $\mathfrak{I}=(r_1,r_2)\subset\mathbb{R}$ be an interval, with $r_2-r_1>2$. Consider a non-empty interval $\mathfrak{I}^*\subset(r_1+1,r_2-1)$ and suppose that the function f is a polynomial of degree $N-1$ \footnote{If we instead fix the parameter $N$ in the difference equation \eqref{eq:frecurrence} but let $d$ be the degree  of the polynomial $f$, the same argument shows that there does not exist a non-trivial polynomial solution of \eqref{eq:frecurrence} with degree $d\neq N-1$.} on $\mathfrak{I}$ and moreover satisfies the following difference equation $\forall x\in \mathfrak{I}^*$:
\begin{multline}
 (2x+4)(2s+2x+3)f(x+1)+2x(2x+1-2s)f(x-1)\\= 2\left(2N(N+2s)+(2x+2)^2\right)f(x).
 \label{eq:frecurrence}
\end{multline}
Then, f is uniquely determined on $\mathfrak{I}^*$ up to multiplication by a constant.
\label{poly2lemma}
\end{lem}
\begin{proof}
Assume that for $\mathbf{a}=(a_{N-1},\ldots,a_0)$ we have
\begin{align*}
f(x)=a_{N-1}x^{N-1}+\ldots+a_{1}x+a_0 , \ \forall x\in \mathfrak{I}.
\end{align*}
Substituting this into the difference equation \eqref{eq:frecurrence} and collecting terms, we get an equation of the form:
\begin{align*}
b_{N+1}x^{N+1}+\ldots+b_{1}x+b_0=0, \ \forall x\in \mathfrak{I}^*.
\end{align*}
Since this equation holds on an interval, it implies that $b_i=0$ for all $i$ and hence we obtain a system of $N+2$ linear equations for $\mathbf{a}$. A simple computation reveals that the contribution to $b_i$ from $a_j$ is zero for $j<i$ and the contribution to $b_i$ from $a_i$ is:
\begin{equation*}
    4i(i-1)+(8s+12)i-4N(N+2s)+8s+4.
\end{equation*}
In particular, when $s>0$, the contribution to $b_i $ from $a_i$ is zero if and only if $i=N-1$.
Hence, we have a vector equation of the form:
\begin{align*}
\mathbf{A}\mathbf{a}=0
\end{align*}
where $\mathbf{A}$ is an $(N+2)\times N$ matrix with first 3 rows zero, and for $i>3$, $\alpha_{i,i-2} \neq 0$ and $\alpha_{ij}=0$ for $j>i-2$,
\begin{equation}
  \mathbf{A} = \begin{bmatrix} 
    0 & 0 & 0 & \ldots & 0 \\
    0 & 0 & 0 & \ldots & 0 \\
    0 & 0 & 0 &\ldots & 0 \\
    \alpha_{4,1} & \alpha_{4,2} & 0 & \ldots & 0\\
    \alpha_{5,1} & \alpha_{5,2} & \alpha_{5,3} & \ldots & 0\\
    \vdots & &  & \ddots & \vdots \\
    \alpha_{N+2,1} & \ldots & &   & \alpha_{N+2,N} 
    \end{bmatrix}.
\end{equation}
In particular, the null space of $\mathbf{A}$ has dimension $1$ and $f$ is uniquely determined on $\mathfrak{I}^*$ up to multiplication by a constant.

\end{proof}

We are finally in a position to prove our main result.

\begin{proof}[Proof of Theorem \ref{MainTheorem}]
From our definitions of $\mathsf{Q}(k;s,N)$ and $\mathsf{J}(k;s,N)$, this is equivalent to proving that for $s\in\mathbb{R}$, $s>0$ and $-\frac{1}{2}<\Re\left( k\right)<s-\frac{1}{2}$ we have:
\begin{align}
\mathsf{J}(k;s,N)=\frac{i^{1-N}}{\Gamma\left(s+\frac{3}{2}\right)\Gamma\left(-s-\frac{1}{2}\right) 2\sqrt{\pi}}  \Gamma\left(\frac{1}{2}-s-N\right)s(2s+N) \times S_{N-1}\left(x;1,\frac{1}{2}+s,1,\frac{1}{2}+s\right).
\label{JMainResult}
\end{align}
For fixed $s>3$, the result follows immediately on the interval $k \in\left(\frac{1}{2},s-\frac{5}{2}\right)$ by noting that the given continuous-Hahn polynomial satisfies the difference equation \eqref{eq:frecurrence} on this interval(see \cite{HypergeometricOrthogonalPolynomials}), upon computing the initial condition and combining Lemmas \ref{poly1lemma} and \ref{poly2lemma}. Since $\mathsf{J}(k;s,N)$ and right hand-side of \eqref{JMainResult} are two functions that are analytic in the strip $-\frac{1}{2}<\Re\left( k\right)<s-\frac{1}{2}$ and they agree on a set which has a limit point in this strip, they agree on the whole strip so that the analytic extension of $\mathsf{J}(k;s,N)$ to the whole of $\mathbb{C}$ is also this polynomial.

We now fix $k\in \mathbb{C}$ such that $\Re(k)+\frac{1}{2}>0$ and remove the restriction $s >3$ to extend the result to the full range $s>\Re(k) + \frac{1}{2}>0$. We begin by noting that from the definition \eqref{eq:pseudojacobipolynomials} of the Pseudo-Jacobi polynomials, for fixed $x,m$ and $N$, the function $s \mapsto p_m^{(s,N)}(x)$ has an analytic continuation to $\left\{ s \in \mathbb{C} : \Re (s) > -\frac{1}{2} \right\} $, given by\footnote{Note that, this is not the same as the pseudo-Jacobi polynomial with complex parameter, which is not holomorphic in $s$.}:
\begin{equation*}
    s \mapsto (x-i)^m\  _2F_1\left[\begin{matrix}
    -m, s+N-m \\
    2s + 2N - 2m   
    \end{matrix} ; \frac{2}{1+ix}\right].
\end{equation*}
Similarly the functions $s \mapsto \frac{d}{dx}p_m^{(s,N)}(x)$, $s \mapsto \phi_N^{(s)}(x)$ and $s \mapsto \gamma_{N-1,s}^2$ can be analytically continued to $\left\{ s \in \mathbb{C} : \Re (s) > -\frac{1}{2} \right\} $. Thus, via an application of Fubini's theorem and Morera's theorem, we see from the formula \eqref{Qtoa(k)}, that the function $s \mapsto \mathsf{Q}(k;s,N)$\footnote{This can also be seen by applying Morera's and Fubini's theorems directly to the $N$-dimensional integral that represents $\mathsf{Q}(k;s,N)$.}, and thus also $\mathsf{J}(k;s,N)$, can be analytically continued to $\left\{ s \in \mathbb{C} : \Re (s) > \Re(k) + \frac{1}{2} \right\} $. Note that the right hand side of \eqref{JMainResult} is analytic as a function of $s$, for $\Re(s) > -\frac{1}{2}$, and by the above result agrees with $\mathsf{J}(k;s,N)$ for $\Re(s) > \max\left(\Re(k)+\frac{1}{2},3\right)$. Hence using the analytic continuation of $ \mathsf{J}(k;s,N)$ we see that the equality \eqref{JMainResult} holds for all $s>\Re(k) + \frac{1}{2}>0$.

Since the polynomials $S_n$ form a family of orthogonal polynomials on a positive weight along the real axis, its zeros are all real (see \cite{CundenMezzadriOConnellSimm}) so that all complex zeros are in the form $k=ix-1$ with $x\in \mathbb{R}$. For the reflection property we simply use the property of continuous-Hahn polynomials $S_n$ that for any $a,b>0$ (see \cite[Theorem 6.6]{CundenMezzadriOConnellSimm} for more details on this):
\begin{equation}
    S_n(-x;a,b,a,b)=(-1)^n S_n(x;a,b,a,b).
\end{equation}
\end{proof}
\begin{rmk}
Using the hypergeometric definition of continuous-Hahn polynomials in \eqref{eq:2} and the fact that for $l<n$,
\begin{align*}
\frac{{(a+b)_n}{(a+c)_n}}{{(a+b)_l}{(a+c)_l}}={(a+b+l)_{n-l}}{(a+c+l)_{n-l}},
\end{align*}
we can see that for fixed $k\in\mathbb{C}$ with $\Re(k)>-\frac{1}{2}$; $\frac{\mathsf{Q}(k;s,N)\Gamma\left(s+\frac{3}{2}\right)\Gamma\left(-s-\frac{1}{2}\right)}{\Gamma\left(s-k-\frac{1}{2}\right)\Gamma\left(\frac{1}{2}-s-N\right)}$ is not only analytic but is also a polynomial in the parameter $s$ for $s>\Re(k)+\frac{1}{2}$.
\end{rmk}

We close this section with the following result on certain moments of the circular ensemble \eqref{eq:haarmeasure}.
\begin{cor}
Let $s>0$ and denote by $e^{i\theta_1}, \ldots, e^{i\theta_N}$ the eigenvalues of a random unitary matrix $\mathbf{U}\in \mathbb{U}(N)$ distributed according to the measure in \eqref{eq:haarmeasure}. Define
\begin{equation}
    \mathsf{T}(k;s,N) = \widetilde{\mathbb{E}}_N^{(s)}\left(\sum_{j=1}^{N}\left|\tan \left(\frac{\theta_j}{2}\right)\right|^{2k}\sec^2\left(\frac{\theta_j}{2}\right)\right),
\end{equation}
where $\widetilde{\mathbb{E}}_N^{(s)}$ denotes the expectation with respect to the measure in \eqref{eq:haarmeasure}, so that $\mathsf{T}(k;s,N)$ exists for $-\frac{1}{2}<\Re\left( k\right)<s-\frac{1}{2}$. Then, we have that:
\begin{align}
    \mathsf{T}(k;s,N) =\frac{\Gamma\left(k+\frac{1}{2}\right) \Gamma\left(s-k-\frac{1}{2}\right)}{\Gamma\left(s+\frac{3}{2}\right)\Gamma\left(-s-\frac{1}{2}\right) \sqrt{\pi}} \frac{i^{1-N}}{2} \Gamma\left(\frac{1}{2}-s-N\right)s(2s+N)
    \times S_{N-1}\left(x;1,\frac{1}{2}+s,1,\frac{1}{2}+s\right),
\end{align}
where $k=ix-1$ and $S_n(x;a,b,c,d)$ denotes the continuous-Hahn polynomial with parameters $a,b,c,d$ and $n$. In particular, $\frac{\mathsf{T}(k;s,N)}{\Gamma\left(k+\frac{1}{2}\right) \Gamma\left(s-k-\frac{1}{2}\right)} $ extends to an analytic function in $\mathbb{C}$ that is invariant up to a change of sign under reflection $k\rightarrow -k-2$, with all complex zeros lying on the vertical line $\Re(k)=-1$.
\begin{proof}
Note that:
\begin{align*}
    \mathsf{Q}(k;s,N)=\mathbb{E}_{N}^{(s)} \left( \Tr \left( |\mathbf{H}|^{2k+2}+|\mathbf{H}|^{2k}\right) \right)
    &=\mathbb{E}_N^{(s)}\left(\sum_{j=1}^{N}\left|\lambda_j\right|^{2k}\left(1+{\lambda_j}^2\right)\right)\\
    &=\widetilde{\mathbb{E}}_N^{(s)}\left(\sum_{j=1}^{N}\left|\tan \left(\frac{\theta_j}{2}\right)\right|^{2k}\sec^2\left(\frac{\theta_j}{2}\right)\right),
\end{align*}
where $\lambda_1, \ldots, \lambda_N$ denote the eigenvalues of $\mathbf{H}$ and the last equality is obtained by applying the inverse of the Cayley transform given in \eqref{cayley}. Now the result is immediate using Theorem \ref{MainTheorem}.
\end{proof}
\end{cor}

\section{A differential equation for the one point density}
In this section we obtain a third order ordinary differential equation (ODE) for the one-point density function $\bm{\rho}_N^{(s)} \left( x\right)$. Such ODEs have been used to prove optimal rates of convergence towards a limit density, see \cite{GotzeTikhomirov}. An analogous ODE was obtained for the GUE, LUE and JUE cases in \cite[§6]{CundenMezzadriOConnellSimm} by using the difference equations for moments\footnote{In the case of the GUE and LUE these ODEs were first obtained by G\"otze and Tikhomirov in \cite{GotzeTikhomirov}.}. We use a similar approach here.

We define the Mellin Transform of $f(x)$ by:
\begin{equation}
    \left[\mathcal{M}f\right](z)=\int _{\mathbb{R^{+}}} x^{z-1}f(x)dx,
\end{equation}
and note that if the integral converges on a strip $D$ of the complex plane we have the following properties, see for example \cite{Mellin}: 
\begin{align}
    \left[\mathcal{M}f^{(m)}\right](z)=(-1)^m(z-m)_m\left[\mathcal{M}f\right](z-m), &&  \text{for} \ z-m \in D, \label{eq:mellin1} \\
    \left[\mathcal{M}\left({x^m}f\right)\right](z)=\left[\mathcal{M}f\right](z+m), &&  \text{for} \ z+m \in D.
    \label{eq:mellin2}
\end{align}

\begin{prop}
Let $s>-\frac{1}{2}$. Then, the one point density function $\bm{\rho}_N^{(s)}(x)$, satisfies the third order differential equation $\mathsf{D}\bm{\rho}_N^{(s)}(x)=0$ where $\mathsf{D}$ is the differential operator given by:
\begin{multline}
(\mathsf{D}y)(x) = (1+x^2)^3 y^{\prime \prime \prime}(x) + 8x(1+x^2)^2 y^{\prime \prime}(x) \\ + 2(1+x^2)(3+2N(N+2s)+(7-2s^2)x^2)y^\prime(x) \\+ 4x(1+s^2+2N(N+2s) + (1-s^2)x^2)y(x).
\end{multline}
\begin{proof}
We note that for $s>3$, from definition of $\mathsf{Q}(k;s,N)$:
\begin{equation*}
\mathsf{Q}(k;s,N)=2\left(\left[\mathcal{M}\bm{\rho}_N^{(s)}\right](2k+3)+\left[\mathcal{M}\bm{\rho}_N^{(s)}\right](2k+1)\right).
\end{equation*}
 Hence, letting $z=2k+2$, by \eqref{eq:13} we get that:
\begin{align*}
(z+2)(4s^2-(z+1)^2)\left(\left[\mathcal{M}\bm{\rho}_N^{(s)}\right](z+3)+\left[\mathcal{M}\bm{\rho}_N^{(s)}\right](z+1)\right)\\
-2(z-1)\left(2N(N+2s)+z^2\right)\left(\left[\mathcal{M}\bm{\rho}_N^{(s)}\right](z+1)+\left[\mathcal{M}\bm{\rho}_N^{(s)}\right](z-1)\right)\\
-(z-1)(z-2)(z-3)\left(\left[\mathcal{M}\bm{\rho}_N^{(s)}\right](z-1)+\left[\mathcal{M}\bm{\rho}_N^{(s)}\right](z-3)\right)=0.
\end{align*}
Taking the inverse Mellin transform of both sides and by applying the properties of the Mellin transform discussed above, we get the desired result for $s>3$. Using the fact that the functions  $ s \mapsto \frac{d^k}{dx^k}p_N^{(s,N)}(x)$, for fixed $x, N$ and $k$, have analytic continuations to $\left\{ s \in \mathbb{C} : \Re (s) > -\frac{1}{2} \right\}$, as discussed in the proof of Theorem \ref{MainTheorem}, we see that the function $s \mapsto \frac{d
^k}{dx^k}\bm{\rho}_N^{(s)}(x)$ may also be analytically continued to this set for $k\in \{0,1,2,3\}$ and thus we can extend the result to $s>-\frac{1}{2}$.
\end{proof}
\end{prop}

\section{The large $N$ limit of the moments}

In this section we compute the large-$N$ limit of the moments. We make use of the following result on the scaled limit of the correlation functions due to Borodin and Olshanski (see \cite[\S 2]{BorodinOlshanskiErgodic}):
\begin{thm}[Borodin \& Olshanski]
Let $\bm{\rho}_N^{(s)}$ be the one-point density function defined as in \eqref{eq:onepointdensity}. Define the scaled one-point density $\rho_N^{(s)}(x):=N\bm{\rho}_N^{(s)}(Nx)$. Then, the scaled limit
$\rho_\infty^{(s)}(x)  := \lim\limits_{N\to\infty}\rho_N^{(s)}(x)$ is given by:
\begin{align*}
\rho_\infty^{(s)}(x) = \frac{1}{2\pi}\Gamma\left[
\begin{matrix}
s+1, & s+1 \\
2s+1, & 2s+2 
\end{matrix}\right]W[Q^{(s)}(x),P^{(s)}(x)],
\end{align*}
for $x>0, \ s > 1/2 $ and the convergence is uniform on compact subsets of $(0,\infty)$. Here $W[f(x),g(x)]$ denotes the Wronskian of two functions $f(x), g(x)$ and $P^{(s)}(x), Q^{(s)}(x)$ are given by:
\begin{equation}
P^{(s)}(x) = 2^{2s-1/2}\Gamma(s+1/2) \cdot x^{-1/2}J_{s-1/2}(1/x),   
\end{equation}
\begin{equation*}
Q^{(s)}(x) = 2^{2s+1/2}\Gamma(s+3/2) \cdot x^{-1/2} J_{s+1/2}(1/x),    
\end{equation*}
where $J_{\nu}$ denotes the Bessel function with parameter $\nu$. 
\
\label{borodinThm}
\end{thm}
The main result of this section is the following:
\begin{thm}
Let $\mathsf{Q}(k;s,N)$ be defined as in \eqref{eq:defnQ}. If $s>\frac{1}{2}$ and $k \in [0,s-1/2)$, then
\begin{equation*}
\lim\limits_{N \to \infty} \frac{\mathsf{Q}(k;s,N)}{N^{2k+2}} = \frac{s\Gamma(k+\frac{1}{2})\Gamma(-\frac{1}{2}-k+s)}{2\sqrt{\pi}\Gamma(k+2)\Gamma(k+\frac{3}{2}+s)}.
\end{equation*}
\label{thm:Qlim}
\end{thm}
The theorem above will follow as a corollary of the following two propositions, which are of independent interest:
\begin{prop}[Interchange of limit] If $s>\frac{1}{2}$ and $ 1<y<2s+1$ then
\begin{equation}
    \lim\limits_{N\to\infty}\frac{1}{2N^y}\mathbb{E}_{N}^{(s)} \left( \Tr  |\mathbf{H}|^{y} \right)= \lim\limits_{N\to\infty}\int_0^\infty x^y \rho_N^{(s)}(x)dx = \int_0^\infty x^y \rho_\infty^{(s)}(x) dx.
\end{equation}
Moreover, if $s>\frac{1}{2}$ and $0<y<2s+1$ then:
\begin{equation}
      \lim\limits_{N\to\infty}\frac{1}{N^{y+2}}\mathbb{E}_{N}^{(s)} \left( \Tr  |\mathbf{H}|^{y} \right) = 0 .
\end{equation}\label{prop:interchange}\end{prop}
\begin{prop}[Explicit calculation of integral] If $s>\frac{1}{2}$ and $1<y<2s+1$ then
\begin{equation}
\int_0^\infty x^y \rho_\infty^{(s)}(x) dx  =   \frac{s\Gamma(\frac{y-1}{2})\Gamma(\frac{1-y}{2}+s)}{4\sqrt{\pi}\Gamma(\frac{y}{2}+1)\Gamma(\frac{1+y}{2}+s)}.      \label{eq:limit1}
\end{equation}
\label{prop:evaluate}
\end{prop}

\begin{proof}[Proof of Theorem \ref{thm:Qlim} assuming Propositions \ref{prop:interchange} and \ref{prop:evaluate}]
Recall the definition:
\begin{equation*}
    \mathsf{Q}(k;s,N)=\mathbb{E}_{N}^{(s)} \left( \Tr \left( |\mathbf{H}|^{2k+2}+|\mathbf{H}|^{2k}\right) \right).
\end{equation*}
It follows immediately from Proposition \ref{prop:interchange} that:
\begin{equation*}
\lim\limits_{N \to \infty} \frac{\mathsf{Q}(k;s,N)}{N^{2k+2}} = \lim\limits_{N\to\infty}\frac{1}{N^{2k+2}}\mathbb{E}_{N}^{(s)} \left( \Tr  |\textbf{H}|^{2k+2} \right) = 2 \int_0^\infty x^{2k+2} \rho_\infty^{(s)}(x) dx.
\end{equation*}
The result is now immediate from the formula \eqref{eq:limit1}.
\end{proof}

\begin{proof}[Proof of Proposition \ref{prop:interchange}]
 First, recall the expression (\ref{eq:onepointdensity}) for the one-point density function $\bm{\rho}_N^{(s)}$. Then, note the estimates, for $x > 1$, $c>b>0$ and $m \in \{0,\dots,N\}$:
\begin{equation*}
    \left\vert_2F_1 \left[\begin{matrix}
    m-N & b \\
    c \end{matrix} \ ; \ \frac{2}{1+iNx}\right]\right\vert \le \sum_{k=0}^{\infty} \frac{(b)_k}{k! (c)_k} \cdot \frac{2^k}{x^k} \le e^2 ,
\end{equation*}
where $_2F_1$  denotes the Gauss hypergeometric function defined as in \eqref{eq:hypergeom}.
Hence using the definition of the pseudo-Jacobi polynomials $p_m^{(s,N)}(x)$ as given in \eqref{eq:pseudojacobipolynomials}, we get that for $s > 0$, $x>1$, 
\begin{multline*}
 |W(p_N^{(s,N)}(Nx),p_{N-1}^{(s,N)}(Nx))\phi_N^{(s)}(Nx)|
\\ \lesssim \left\{ 
|Nx-i|^{-2} + N|Nx-i|^{-3}
+ |Nx-i|^{-2} + (N-1)|Nx-i|^{-3}
\right\}\cdot |Nx-i|^{-2s}\\
\lesssim N^{-2s-2} \cdot x^{-2s-2},
\end{multline*}
where we use here (and in the remainder of this proof) the notation $\lesssim$ to denote an inequality up to an implicit constant which is independent of $N$ and $x$.
As $\Gamma(2s+N+1)/\Gamma(N) \sim N^{2s+1}$, using \eqref{eq:onepointdensity} and recalling that $\rho_N^{(s)}(x)=N\bm{\rho}_N^{(s)}(Nx)$ we obtain:
\begin{align*}
    \rho_N^{(s)}(x) \lesssim x^{-2s-2}.
\end{align*}
Therefore, when $s > \frac{1}{2}(y-1)$, by dominated convergence theorem we conclude
\begin{equation*}
    \lim\limits_{N\to\infty}\int_1^\infty \rho_N^{(s)}(x) \cdot x^y dx = \int_1^\infty \rho_\infty^{(s)}(x) \cdot x^y dx.
\end{equation*}
As $\rho_N^{(s)}$ converges uniformly on compact subsets of $(0,\infty)$, we get that for any $\epsilon > 0$,
\begin{equation*}
    \lim\limits_{N\to\infty}\int_\epsilon^\infty \rho_N^{(s)}(x) \cdot x^y dx = \int_\epsilon^\infty \rho_\infty^{(s)}(x) \cdot x^y dx.
\end{equation*}
Therefore it will be sufficient to prove that
\begin{equation}
\begin{cases}
     \sup\limits_{N \ge 1} \int_0^\epsilon \rho_N^{(s)}(x) \cdot x^y dx \lesssim \epsilon^{y-1} \ ,  & \text{whenever} \ {y>1}, \label{eq:cond1}\\
     & \\
     \lim\limits_{N\to\infty} \frac{1}{N^2} \int_0^\epsilon \rho_N^{(s)}(x) \cdot x^y dx = 0 \ ,  & \text{whenever} \ y \in [0,1].
\end{cases}
\end{equation}

Adapting a method due to Y. Qiu from \cite{Qiu}, we transform the random point configuration associated to the generalized Cauchy ensemble onto the unit circle, via the Cayley transform  $x \mapsto e^{i\theta} = (i-x)/(i+x)$ (recall \eqref{eq:unitcircledensity}). Arguing analogously to \cite[\S 2.2.1, \S 2.2.2]{Qiu}, we can rewrite the integral in \eqref{eq:cond1}:
\begin{equation}
    \int_0^\epsilon \rho_N^{(s)}(x) \cdot x^y dx  = \frac{1}{N^y} \int_0^{2\arctan(N\epsilon)} (\tan \theta/2)^y \Tilde{\rho}_N^{(s)}(\theta) d \theta
    \label{eq:changevariable}
\end{equation}
where $ \Tilde{\rho}_N^{(s)}(\theta)$ is the one-point density function on $[0,2\pi)$ given by 
\begin{equation}
    \Tilde{\rho}_N^{(s)}(\theta) = \lambda^{(s)}(e^{i\theta}) \sum_{n=0}^{N-1}|q_n^{(s)}(e^{i\theta})|^2. 
    \label{eq:circle}
\end{equation}

Here $\lambda^{(s)}(e^{i\theta})$ is proportional to $|1+e^{i\theta}|^{2s}$ and scaled so that $\lambda^{(s)}(e^{i \theta}) \frac{d\theta}{2\pi}$ is a probability measure on $(0,2 \pi)$, and $\left(q_n^{(s)}\right)_{n \ge 1}$ is the system of polynomials with $\deg q_n^{(s)} = n $, orthonormal with respect to the inner product $\langle f, g \rangle = \int_0^{2\pi} f(e^{i\theta}) \overline{g(e^{i\theta})} \lambda^{(s)}(e^{i\theta}) \frac{d\theta}{2\pi}$. By \cite{Golinskii}, (see also \cite[\S 2.2.4]{Qiu}), there exists the following estimate for $q_n^{(s)}(e^{i\theta})$:
\begin{equation*}
    \lambda^{(s)}(e^{i\theta})\cdot |q_n^{(s)}(e^{i\theta})|^2 \lesssim \left( 1+\frac{1}{(n+2)|1+e^{i\theta}|}\right)^{-2s}.
\end{equation*}
Hence, as in \cite[\S 2.2.4]{Qiu} by the formula \eqref{eq:circle} we obtain the estimate
\begin{equation*}
     \Tilde{\rho}_N^{(s)}(\theta) \lesssim N,
\end{equation*}
uniformly in $\theta$, whenever $s>0$. Combining this estimate with the equality \eqref{eq:changevariable} we obtain the following chain of estimates:
\begin{align*}
     \int_0^\epsilon \rho_N^{(s)}(x) \cdot x^y dx 
     \lesssim \frac{1}{N^{y-1}}\int_0^{2\arctan(N\epsilon)}|\tan(\theta/2)|^y d \theta 
     = 2 \int_0^{\epsilon} \frac{x^y}{N^{-2}+x^2} dx.
\end{align*}
Note that, we can estimate:
\begin{equation*}
         2 \int_0^{\epsilon} \frac{x^y}{N^{-2}+x^2} \le \begin{cases}
         \frac{2}{y-1}\epsilon^{y-1}, & \text{if} \ y>1, \\
         2\epsilon^yN\arctan(N\epsilon) \le \pi N \epsilon^y, & \text{if}\  y \in [0,1].
         \end{cases}
\end{equation*}
Thus the conditions from \eqref{eq:cond1} are established, and this completes the proof of Proposition \ref{prop:interchange}.
\end{proof}

\begin{proof}[Proof of Proposition \ref{prop:evaluate}]
We now explicitly calculate the integral $\int_0^\infty x^y \rho_\infty^{(s)}(x) dx$. 
By Theorem \ref{borodinThm}, we get:
\begin{align*}
\rho_\infty^{(s)}(x) = \frac{1}{2 \pi} \Gamma \left[\begin{matrix}
s+1,&s+1,&s+1/2,&s+3/2\\
&2s+1,&2s+2&
\end{matrix}
\right] \cdot\frac{2^{4s}}{x}\cdot W[J_{s+1/2}(1/x),J_{s-1/2}(1/x)].
\end{align*}
Note that, by Legendre's duplication formula, we have:
\begin{align*}
\frac{1}{2 \pi} \Gamma \left[\begin{matrix}
s+1,&s+1,&s+1/2,&s+3/2\\
&2s+1,&2s+2&
\end{matrix}
\right] \cdot 2^{4s} = \frac{1}{4}.
\end{align*}
Expanding the Wronksian and making the change of variables $x \mapsto 1/x$, we get that 

\begin{multline*}
\int_0^\infty \rho_\infty^{(s)}(x) \cdot x^y dx
= \frac{1}{4}\int_0^\infty x^{1-y} \cdot \left[ J_{s+1/2}^\prime(x) J_{s-1/2}(x) - J_{s+1/2}(x)J_{s-1/2}^\prime(x) \right]dx\\
= \frac{1}{8}\int_0^\infty x^{1-y}\left[J_{s+1/2}(x)^2+J_{s-1/2}(x)^2-J_{s+3/2}(x)J_{s-1/2}(x) -J_{s-3/2}(x)J_{s+1/2}(x)\right]dx,    
\end{multline*}
where in the last line we used the identity $J_\nu^\prime(x) = \frac{1}{2}(J_{\nu-1}(x)-J_{\nu+1}(x))$.
We now apply the following identity valid for $\Re(\mu + \nu +1) > \Re(\lambda) > 0$ (see \cite[p.403]{Watson}):

\begin{multline*}
 \int_0^{\infty} \frac{J_\mu(x)J_\nu(x)}{x^\lambda} dx  
 = 2^{-\lambda} \cdot \Gamma \left[ \begin{matrix}
 \lambda, & \frac{1}{2} \mu + \frac{1}{2}\nu -\frac{1}{2}\lambda + \frac{1}{2} \\
 \frac{1}{2} \lambda + \frac{1}{2}\nu -\frac{1}{2}\mu + \frac{1}{2}, & \frac{1}{2} \lambda + \frac{1}{2}\mu + \frac{1}{2}\nu + \frac{1}{2}, & \frac{1}{2} \lambda + \frac{1}{2}\mu -\frac{1}{2}\nu + \frac{1}{2}
 \end{matrix}\right].
\end{multline*}
Therefore we get that
\begin{multline*}
\int_0^\infty \rho_\infty^{(s)}(x) \cdot x^y dx
= \left(\frac{1}{2}\right)^{y+2} \cdot \Gamma(y-1) \cdot \Bigg\{ \frac{\Gamma(s-\frac{1}{2}y+\frac{3}{2})}{\Gamma(\frac{1}{2}y)^2\Gamma(\frac{y+1}{2} + s)}\\
+ \frac{\Gamma(s-\frac{1}{2}y+\frac{1}{2})}{\Gamma(\frac{1}{2}y)^2\Gamma(\frac{y-1}{2} + s)} - \frac{\Gamma(s-\frac{1}{2}y+\frac{3}{2})}{\Gamma(\frac{y}{2}-1)\Gamma(\frac{y+1}{2} + s)\Gamma(\frac{y}{2}+1)} - \frac{\Gamma(s-\frac{1}{2}y+\frac{1}{2})}{\Gamma(\frac{y}{2}-1)\Gamma(\frac{y-1}{2} + s)\Gamma(\frac{y}{2}+1)} \Bigg\} \\
= \frac{s\Gamma(\frac{y-1}{2})\Gamma(\frac{1-y}{2}+s)}{4\sqrt{\pi}\Gamma(\frac{y}{2}+1)\Gamma(\frac{1+y}{2}+s)},
\end{multline*}
where the last line is obtained by applying gamma function identities. This completes the proof of Proposition \ref{prop:evaluate}.
\end{proof}

\bigskip
\noindent
{\sc School of Mathematics, University of Edinburgh, James Clerk Maxwell Building, Peter Guthrie Tait Rd, Edinburgh EH9 3FD, U.K.}\newline
\href{mailto:theo.assiotis@ed.ac.uk}{\small theo.assiotis@ed.ac.uk}

{\sc Mathematical Institute, Andrew Wiles Building, University of Oxford, Radcliffe
Observatory Quarter, Woodstock Road, Oxford, OX2 6GG, UK.}\newline
\href{mailto:benjamin.bedert@sjc.ox.ac.uk}{\small benjamin.bedert@sjc.ox.ac.uk}

{\sc Mathematical Institute, Andrew Wiles Building, University of Oxford, Radcliffe
Observatory Quarter, Woodstock Road, Oxford, OX2 6GG, UK.}\newline
\href{mailto:mustafa.gunes@st-hildas.ox.ac.uk}{\small mustafa.gunes@st-hildas.ox.ac.uk}

{\sc Mathematical Institute, Andrew Wiles Building, University of Oxford, Radcliffe
Observatory Quarter, Woodstock Road, Oxford, OX2 6GG, UK.}\newline
\href{mailto:arun.soor@sjc.ox.ac.uk}{\small arun.soor@sjc.ox.ac.uk}

\end{document}